\documentclass[12pt,a4paper,twoside,reqno,tbtags]{amsart}
\usepackage{graphicx}
\usepackage{amsmath,dsfont}
\usepackage{amssymb}
\usepackage{amsthm}
\usepackage{amsfonts}
\usepackage{mathrsfs}
\usepackage{float}
\usepackage{tikz}
\usetikzlibrary{shapes,fit}
\usetikzlibrary{shapes.multipart}
\usetikzlibrary{positioning}
\usetikzlibrary{shapes,arrows,backgrounds,fit,positioning}
\usepackage{pgfplots}
\usepackage{adjustbox}
\allowdisplaybreaks

\theoremstyle{plain}
\newtheorem{Theorem}{Theorem}[section]
\theoremstyle{proposition}
\newtheorem{proposition}{Proposition}
\theoremstyle{lemma}
\newtheorem{lemma}{Lemma}
\theoremstyle{remark}
\newtheorem{remark}{Remark}
\theoremstyle{definition}
\newtheorem{definition}{Definition}

\begin{document}
	
	\title []{ A Characterization of Zero  Divisors and Topological Divisors of Zero in $\mathbf{C[\MakeLowercase{a,b}]}$ and $\ell^\infty$ }
	
	%\author{H. Chandra, Anurag K. Patel}
	
	\author[]{Harish Chandra}
	\address{Harish Chandra \\Department of Mathematics \\ Banaras Hindu University \\ Varanasi 221005, India}
	\email{harishc@bhu.ac.in}
	
	\author[]{Anurag Kumar Patel}
	\address{Anurag Kumar Patel \\ Department of Mathematics \\ Banaras Hindu University \\ Varanasi 221005, India}
	\email{anuragrajme@gmail.com}

%	\CorrespondingAuthor{Anurag Kumar Patel}
	
	%\date{\today}
	

	\begin{abstract}
	We give a characterization of zero divisors of the ring $C[a,b].$ Using the Weierstrass approximation theorem, we completely characterize topological divisors of zero of the Banach algebra $C[a,b].$ We also characterize the zero divisors and topological divisors of zero in $\ell^\infty.$  Further,  we show that zero is the only zero divisor in the disk algebra $\mathscr{A}(\mathbb{D})$ and that the class of singular elements in $\mathscr{A}(\mathbb{D})$ properly contains the class of topological divisors of zero. Lastly, we construct a class of topological divisors of zero of $\mathscr{A}(\mathbb{D})$ which are not zero divisors.	
	\end{abstract}
	
	\maketitle
	
\section{Introduction}

Throughout this paper, $\mathbb{N}$ denotes the set of all natural numbers, $\mathbb{C}$ denotes the set of complex numbers, $C[a,b]$ denotes the Banach algebra of all continuous complex valued functions on the closed interval $[a,b]$ under the supremum norm. Further, $\ell^\infty$ denotes the Banach algebra of all bounded sequences of complex numbers, $\mathcal{C}_{0}$ denotes the space of all sequences of complex numbers converging to $0$ and $\mathcal{C}_{00}$ denotes the space of all sequences of complex numbers whose all but finitely many terms are zero.  Let $\mathbb{D}=\{z\in \mathbb{C}:|z|<1\},$ $\mathbb{\bar{D}}$ be its topological closure and  $\mathbb{T}=\{z\in \mathbb{C}:|z|=1\}$ denote the unit circle. Let $\mathscr{A}(\mathbb{D})$ denote the disk algebra, the sup-normed Banach algebra of functions continuous on $\mathbb{\bar{D}},$ which are analytic in $\mathbb{D}.$ 

\begin{definition}[Zero Set]\label{def1}
	Let $f \in C[a,b]$. Then the zero set of $f$ is the set defined by 
	$$Z_f=\{x\in [a,b] : f(x)=0\}.$$
\end{definition}
\begin{lemma}\label{shiker}
	Let $f \in C[0,1]$. Then the zero set of $f$ is a closed set.
	\begin{definition}(\cite{Simmons})\label{anurag1}
		Let $\mathcal{A}$ be a Banach algebra. An element $x\in \mathcal{A}$ is said to be regular if there exists an element $y\in \mathcal{A}$ such that $ xy=yx=1.$ An element $x\in \mathcal{A}$ is  singular if it is not regular. 
	\end{definition}
	\begin{definition}
		A sequence $(x_n)_{n=1}^\infty$ of complex numbers is said to be ``bounded away from zero" if there exists a positive constant $\delta>0$ so that $|x_n|\geq\delta$ for all $n\in{\mathbb{N}}.$
	\end{definition}
\end{lemma}
\begin{lemma}\emph{(\cite{Kreyszig})}\label{shiker2}
	Let $A$ be a subset of a metric space $(X,d)$. Then the following statements are equivalent:
	\begin{enumerate}
		\item[$(1)$] $A$ is nowhere dense.
		\item[$(2)$] $\bar{A}$ does not contain any non-empty open set.
	\end{enumerate}
\end{lemma}
\begin{lemma}\label{shiker3}
	Let $(X,d)$ be a metric space. If $A $ is a closed nowhere dense subset of $X$, then the complement $A^c$ of $A$ is an open dense set.
\end{lemma}
\begin{lemma}\emph{(\cite{Kreyszig})}\emph{[Closure, Closed Set]}\label{shiker4}
	Let $M$ be a nonempty subset of a metric space $(X,d)$ and $\overline{M}$ be  its closure, then 
	\begin{enumerate}
		\item[$(1)$] $x\in \overline{M}$ if and only if there is a sequence $(x_n)_{n=1}^\infty$ in $M$ such that $x_n\to x$ as $n\to \infty$.
		\item[$(2)$] $M$ is closed if and only if the situation $x_n \in M$, $x_n \to x$ as $n\to \infty$  implies that $x\in M$.
	\end{enumerate} 
\end{lemma}
\begin{Theorem}\emph{(\cite{Rudin})}\emph{[The Weierstrass Approximation Theorem]}
	If $f$ is a continuous complex function on $[a,b]$, and $\epsilon>0$ is given. Then there exists a polynomial $p$ such that 
	$$|f(x)-p(x)|<\epsilon \mbox{ for all } x\in [a,b].$$ 
\end{Theorem}
\begin{definition}(\cite{Simmons}){[Zero Divisors]}
	Let $R$ be a ring. Then an element $z\in R$ is said to be a zero divisor if  either $zx=0$ for some non-zero $x \in R$ or $yz=0$ for some non-zero $y \in R$.
\end{definition}
\begin{definition}(\cite{Conway,Simmons})[Topological Divisors of Zero]
	An element $z$ in a Banach algebra $\mathcal{A}$ is called a topological divisor of zero if there exists a sequence $(z_n)_{n=1}^\infty$ in $\mathcal{A}$ such that \begin{enumerate}
		\item[$(1)$] $\|z_n\|=1 \quad \forall\ n \in \mathbb{N}$;
		\item[$(2)$] Either $zz_n \to 0$ or $z_nz\to 0$ as $n\to \infty$.
	\end{enumerate}
\end{definition}
We give a proof of the following lemma for the sake of completeness.
\begin{lemma}\label{shiker1}
	The set of all topological divisors of zero in a Banach algebra is a closed set.
\end{lemma}
\begin{proof}
	Let $\mathcal{A}$ be a Banach algebra. Define $\varphi:\mathcal{A}\to [0,\infty)$ as $$\varphi(a)=\inf_{\|b\|=1}\|ab\|~~ \forall a\in \mathcal{A}.$$
	Then we observe that $a$ is a topological divisor of zero if and only if $\varphi(a)=0.$ To get the desired conclusion, it is sufficient to prove that $\varphi$ is continuous. To this end, let $(a_n)_{n=1}^\infty$ be a sequence in $\mathcal{A}$ such that $a_n \to a$ as $n\to \infty.$ Let $\epsilon>0.$ Then there exists $b\in \mathcal{A}$ with $\|b\|=1$ such that 
	\begin{equation}\label{shailesh1}
		\varphi(a)\leq\|ab\|<\varphi(a)+\epsilon.
	\end{equation}
	Further, we also have $\varphi(a_n)\leq\|a_nb\|$ for all $b$ with $\|b\|=1$ and $\mbox{ for all } n\geq 1.$ This together with \eqref{shailesh1} implies that
	$$\limsup_{n\to \infty}\varphi(a_n) \leq \limsup_{n\to \infty}\|a_nb\|=\lim_{n\to \infty}\|a_nb\|=\|ab\|<\varphi(a)+\epsilon,$$
	as $\epsilon$ is arbitrary, we get that $\underset{n\to\infty }{ \limsup} ~\varphi(a_n)\leq\varphi(a).$

	Next, let $\epsilon>0.$ Pick a sequence $(b_n)_{n=1}^\infty$ in $\mathcal{A}$ with $\|b_n\|=1$ such that 
	\begin{equation}\label{shailesh2}
		\|a_nb_n\|<\varphi(a_n)+\epsilon~~~\forall n\geq 1.	
	\end{equation}	
	\par
	Also, we have 
	
	$$|\|a_nb_n\|-\|ab_n\||\leq\|(a_n-a)b_n\|\leq\|a_n-a\| \to 0 \mbox{ as } n\to\infty.$$
	This gives that for sufficiently large  $n,$ we have $\|ab_n\|-\epsilon<\|a_nb_n\|<\|ab_n\|+\epsilon,$
	This together with \eqref{shailesh2} gives that 
	$$\varphi(a)\leq\|ab_n\|<\|a_nb_n\|+\epsilon<\varphi(a_n)+2\epsilon,$$
	as $\epsilon$ is arbitrary, the preceding inequality gives that $\varphi(a)\leq \underset{n\to\infty }{\liminf}\varphi(a_n).$
	Thus, we must have $\underset{n\to\infty}{\lim}\varphi(a_n)=\varphi(a).$ This completes the proof.
	
\end{proof}

S.J Bhatt, H.V.Dedania (\cite{Bhatt}) proved the following result.
\begin{Theorem}
	Every element of a complex Banach algebra $(\mathcal{A},\|\cdot\|)$ is a topological divisor of zero (TDZ), if at least one of the following holds:
	\begin{enumerate}
		\item[$(1)$] $\mathcal{A}$ is infinite dimensional and admits an orthogonal basis.
		\item[$(2)$] $\mathcal{A}$ is a nonunital uniform Banach algebra ($u\mathcal{B}$-algebra) in which the Silov boundary $\partial \mathcal{A}$ coincides with the carrier space (the Gelfand space) $\Delta(\mathcal{A})$ (in  particular, $\mathcal{A}$ is nonunital regular $u\mathcal{B}$-algebra).
		\item[$(3)$] $\mathcal{A}$ is a nonunital hermitian $Banach^\ast$-algebra with continuous involution (in particular, $\mathcal{A}$ is a nonunital $\mathcal{C}^\star-$algebra). 	
	\end{enumerate}
\end{Theorem}
Motivated by the above theorem, we characterize zero divisors and topological divisors of zero in $C[a,b]$ and $\ell^\infty.$ We also show that zero is the only zero divisor in $\mathscr{A}(\mathbb{D}).$ Further, we give a class of singular elements of $\mathscr{A}(\mathbb{D}),$ which are not topological divisors. Finally, we construct a class of topological divisors of zero in $\mathscr{A}(\mathbb{D}),$ which are not zero divisors. Several results of this paper are new and methods of proof of all the results given in this paper are new and interesting to the best of our knowledge and understanding.

\section{A characterization of Zero divisors and Topological divisors of zero in the Banach algebra $C[a,b]$}
The following theorem gives a complete characterization of zero divisors of $C[a,b]$. 
\begin{Theorem}
	An element $f\in C[a,b]$ is a zero divisor if and only if  zero set of $f$ contains a non-empty open interval.
\end{Theorem}
\begin{proof}
	Let $f\in C[a,b]$ and let $Z_f=\{x\in [a,b] : f(x)=0\}$ be the zero set of $f$ which contains a non-empty open interval $(c,d)$.
	
	Define $g:[a,b]\to\mathbb{R}$ by
	$$g(x)=\begin{cases}
		0,&\mbox{ if } x \in [a,b]\setminus (c,d);\\
		x-c,&\mbox{  if } c<x\leq \frac{c+d}{2};\\
		d-x,&\mbox{ if } \frac{c+d}{2}\leq x<d.
	\end{cases}$$

	\begin{figure}[H]
		\begin{adjustbox}{center}
			\begin{tikzpicture}[scale=0.56]
				\draw[->,color=gray,ultra thick] (-9,0) -- (9,0) node[right] {\textbf{\large x-axis}};
				\draw[-,color=gray,ultra thick] (0,0) -- node[below=2.5cm]{}(0,6)  {};
				\draw[color=red,ultra thick][domain=-5:0]   plot (\x,{5+\x})   node[pos=] {};
				\draw[color=red,ultra thick][domain=0:5]   plot (\x,{5-\x})   node[pos=] {};
				\draw[-,color=red,ultra thick] (-8,0) -- (-5,0) node[right] {};
				\draw[-,color=red,ultra thick] (5,0) -- (8,0) node[right] {};
				\draw[dashed, thick,black] (0, 0) -- node[below=0.3cm]{$\frac{c+d}{2}$} (0,0) ;
				\draw[dashed, thick,black] (-5, 0) -- node[below=0.3cm]{$c$} (-5,0) ;
				\draw[dashed, thick,black] (5, 0) -- node[below=0.3cm]{$d$} (5,0) ;
				\draw[dashed, thick,black] (-8, 0) -- node[below=0.3cm]{$a$} (-8,0) ;
				\draw[dashed, thick,black] (8, 0) -- node[below=0.3cm]{$b$} (8,0) ;
				\draw[dashed, thick,black] (0,5) -- node[left=0.3cm]{$\frac{d-c}{2}$} (0,5) ;
				%	\node[draw=blue,thick,dashed, rounded corners,fit=(-9,0) (9,0) (0,0) (0.6)](1) {};
			\end{tikzpicture}
		\end{adjustbox}
		\caption{Graph of the function $g$}
	\end{figure}
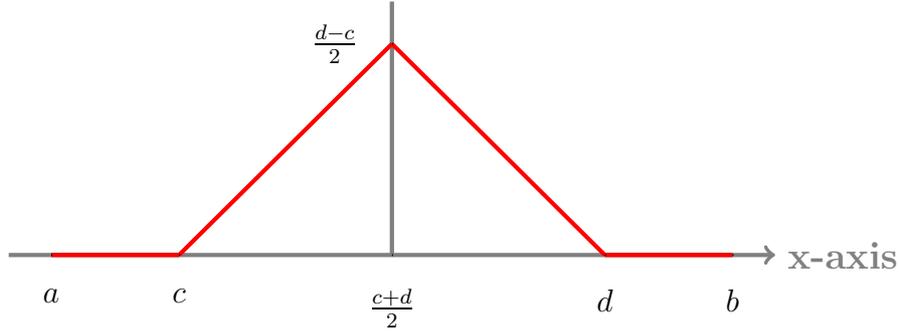
	
	Clearly $ g(x)\neq 0$ on $(c,d)\subseteq [a,b] $ and is a continuous function on $[a,b]$, hence $g\in C[a,b].$
	
	Since $f(x)=0$  on  $ Z_f$, and $ g(x)=0$  on  $V= [a,b]\setminus (c,d)$, then $(fg)(x)=0 \ \forall\ x\in[a,b]$. This shows that $f$ is a zero divisor of $C[a,b]$.
	
	Conversely,	let $f\in C[a,b]$ be a zero divisor. 
	Now suppose  $0\neq f\in C[a,b]$  and on the contrary, assume that $Z_f$  does not contain any non-empty open interval. Then by Lemma $\ref{shiker}$ and Lemma $\ref{shiker2}$, $Z_f$ is a closed nowhere dense set.
	Let $V_f=[a,b]\setminus Z_f$, then by Lemma $\ref{shiker3}$, $V_f$ is an open dense set in $[a,b]$.
	Since $f$ is a zero divisor, there exists  $0\ne g\in C[a,b]$ such that  $(fg)(x)=0 \ \forall\ x\in [a,b]$. Since $f\neq 0$ on $V_f$, so $g(x)=0 \ \forall \ x\in V_f$. 	
	
	Since $V_f$ is an open dense set in $[a,b]$, then from Lemma $\ref{shiker4}$, for each $x\in[a,b]$, there exists  a  sequence $(x_n)_{n=1}^\infty$ in $V_f$ such that   $x_n\to x$ as $n\to \infty$. But $x_n\in V_f$, so $g(x_n)=0 \ \forall\ n\in\mathbb{N}$. Since $g$ is continuous on $[a,b]$, then $g(x)=0.$ Thus $g=0$, which is a contradiction. Hence $Z_f$ must contains a non-empty open interval.
\end{proof}

\begin{lemma}\label{Le5}
	Let $\mathcal{A}$ be a commutative Banach algebra and $x\in \mathcal{A}$ be a topological divisor of zero. Then for each $y\in \mathcal{A},\ xy$ is also a topological divisor of zero. 
	
\end{lemma}
\begin{proof}
	Let $x\in \mathcal{A}$ be the topological divisor of zero. Then there exists a sequence	$(x_n)_{n=1}^\infty$ in $\mathcal{A}$ such that $\|x_n\|=1$, for all n$\in \mathbb{N}$ and $xx_n\to 0$ as $n\to\infty$.
	Let $y\in \mathcal{A}$ be any element. Then, we have  
	$$\|yxx_n\|\leq\|y\|\|xx_n\|.$$
	Since $xx_n\to0$ as $n\to\infty$, then
	$$\|(yx)x_n\| \to 0.$$
	Hence 
	$yx$ is a topological divisor of zero.
\end{proof}
The following theorem gives a complete characterization of the topological divisors of zero in $C[a,b]$.
\begin{Theorem}
	An element $f\in C[a,b]$ is a topological divisor of zero if and only if $f$ has at least one zero in $[a,b].$ 
\end{Theorem}
\begin{proof}
	Let $f\in C[a,b]$ which has a zero, say $f(c)=0$ for some $c\in[a,b]$. Since $f$ is continuous, by the Weierstrass  approximation theorem,
	for given $\epsilon >0,$ there exists a polynomial $p(x)$ such that
	\begin{equation*}
		|f(x)-p(x)|<\epsilon/2 \quad \forall ~ x\in[a,b]
	\end{equation*}
	This implies $$|f(c)-p(c)|<\epsilon/2,$$
	Thus $$|p(c)|<\epsilon/2.$$
	Consider the polynomial $q(x)=p(x)-p(c)$. Then $q(c)=0$
	and $$|f(x)-q(x)|=|f(x)-p(x)+p(c)|\leq |f(x)-p(x)|+|p(c)|<\frac{\epsilon}{2}+\frac{\epsilon}{2}=\epsilon.$$
	Hence we can find a sequence of polynomials $(q_n)_{n=1}^\infty$ in $C[a,b]$ such that $q_n(c)=0 \ \forall \ n\in \mathbb{N}$  and $q_n\to f$ uniformly on $[a,b].$ 
	
	Since  $q_n(c)=0$, $ q_n(x)=(x-c)r_n(x),$
	where $r_n(x)$ is a polynomial in $C[a,b]$.\\
	Now $z(x)=x-c$ is a topological divisor of zero, therefore by the Lemma $\ref{Le5}$, $q_n$ is a topological divisor of zero for all $n\in \mathbb{N}.$ Since $q_n\to f$ uniformly and by Lemma \ref{shiker1}, the class of topological divisors of zero is a closed set, it follows that $f$ is a topological divisor of zero.
	
	Conversely,
	suppose $f\in C[a,b] $ is a topological divisor of zero. Suppose that $f$ has no zero in $[a,b]$. Then, $\frac{1}{f} \in C[a,b]$. Let $g(x)=\frac{1}{f(x)}$, then $g(x)f(x)=1 \ \forall \ x\in [a,b]$. Since $f$ is a topological divisor of zero, there exists a sequence $(f_n)_{n=1}^\infty$ in $C[a,b]$ with $\|f_n\|=1 \ \forall \ n \in\mathbb{N},$ such that  $ff_n \to 0$ as $n\to\infty$. 
	Since $gf=1$, then, $ f_n=gff_n \to 0$ as $n\to\infty$. This is a  contradiction as $\|f_n\|=1\ \forall\ n\in\mbox{N}$. Hence $f$ must have a zero in $ [a,b].$
\end{proof} 
\begin{remark}
	The above theorem shows that $z(t)=(t-c)^k$ is a topological divisor of zero but is not a zero divisor for each $k>0$ and for each $c\in[a,b]$.
\end{remark}
\section{A characterization of Zero divisors and Topological divisors of zero in the Banach algebra $\ell^\infty$}
In this section, we give a complete characterization of regular elements, zero divisors and topological divisors of zero in the Banach algebra $\ell^\infty.$
\begin{Theorem}
	An element $x=(x_n)_{n=1}^\infty \in \ell^\infty$ is a regular element if and only if $x$ is bounded away from zero.
\end{Theorem}
\begin{proof}
	Let $x=(x_n)_{n=1}^\infty \in \ell^\infty$ be a regular element, then there exists an element $y=(y_n)_{n=1}^\infty$ in $\ell^\infty$ such that $xy=(1,1,...,1,...)=1$. That is $x_ny_n=1$ for all $n\in{\mathbb{N}}.$ This implies that, $y_n=\frac{1}{x_n} ~\forall~ n\in{\mathbb{N}}.$ Since $y \in \ell^\infty,$ $\exists~ M>0$ such that $|y_n|\leq M~  \forall~ n\in\mathbb{N}$. Hence $\frac{1}{M}\leq|x_n| ~\forall~  n\in{\mathbb{N}}.$ Hence $x$ is bounded away from zero. \\
	Conversely, let $x\in \ell^\infty$ be bounded away from zero. Then there exists a positive constant $M$ such that $M\leq |x_n|$ for all $n\in{\mathbb{N}}.$  This implies That $\frac{1}{|x_n|}\leq \frac{1}{M}~ \forall~ n\geq 1.$ Now choosing $y=(\frac{1}{x_n})_{n=1}^\infty,$ we get $y=(y_n)\in \ell^\infty$ and $xy=1$. Hence $x$ is a regular element of $\ell^\infty$.
\end{proof}

The following theorem characterizes zero divisors of $\ell^\infty.$
\begin{Theorem}
	An element $(x_n)_{n=1}^\infty \in \ell^\infty,$ is a zero divisor if and only if $\exists$ $n\geq1$ such that $x_n=0$.
\end{Theorem}
\begin{proof}
	Let $x=(x_n)_{n=1}^\infty \in \ell^\infty$ be a zero divisor, then $\exists~ 0\neq$ $y={(y_n)}_{n\geq1}\in \ell^\infty$ such that $xy=(x_ny_n)_{n=1}^\infty=0.$ That is $x_ny_n=0$ $\forall$ $n\in{\mathbb{N}}.$ Since $y\neq0$ then $\exists~ k\geq1$ such that $y_{k}\neq0.$ Therefore, $x_{k}y_{k}=0$ implies that $x_k=0.$\\
	Conversely, let $\exists~ n\geq1$ such that $x_n=0.$ Then for $y=(y_k)_{k=1}^\infty,$ where $y_n=1$ and $y_k=0~ \forall k\neq n,$ we get, $xy=0.$ Hence $x$ is a zero divisor.
\end{proof}
\begin{remark}
	$\mathcal{C}_{00}$ is properly contained in the set of all zero divisors of $\ell^\infty.$
\end{remark}
\begin{proof}
	Let $x=(x_k)_{k=1}^\infty \in \mathcal{C}_{00}$ where $x_k=0~ for~all~k\geq n+1.$ Take $y=(y_k)_{k=1}^\infty$ where $y_k=0$ for all $k\leq n$ and $y_k=1$ for all $k\geq n+1.$ Then $xy=0.$ So $x$ is a zero divisor. Also, note that $x=(0,1,1,...)$ is a zero divisor but not in $\mathcal{C}_{00}.$ So the Inclusion is proper. 
\end{proof}

\begin{Theorem}
	In the Banach algebra $\ell^\infty$ the set of all topological divisors of zero and the set of all singular elements coincide.
\end{Theorem}
\begin{proof}
	Clearly, a topological divisor of zero is a singular element.	Let $x=(x_n)_{n=1}^\infty$ be a singular element in $\ell^\infty.$  Then $x$ is not bounded away from zero. Hence, there exists a subsequence $(x_{n_k})_{k=1}^\infty$ of $(x_n)_{n=1}^\infty$ such that ${x_{n_k}}\to 0$ as $k\to \infty.$ Take $z^{(k)}=e_{n_k}~ \forall~ k\geq1.$ Then $\|z^{(k)}\|=1 ~\forall~ k\geq1$ and $\| xz^{(k)} \|=|x_{n_k}| \to 0$ as $k\to \infty.$ Thus $\|xz^{(k)}\| =|x_{n_k}|\to 0$ as $k\to \infty.$ This shows that $x$ is a topological divisor of zero. Hence the proof.
\end{proof}

\begin{remark}
	$\mathcal{C}_0$ is properly contained in the set of all topological divisors of zero of $\ell^\infty.$
\end{remark}
\begin{proof}
	Let $x=(x_n)_{n=1}^\infty \in \mathcal{C}_0 .$ Then $|x e_n|=|x_n|$ $\to 0$ as $ n\to \infty.$
	Then  $|x_n| \to 0$ as $ n\to \infty.$ So $x$ is a topological divisor of zero. For the proper containment, take the element $x=(x_n)=(0,1,1,...) \in \ell^\infty,$  which is a topological divisor of zero but $x\notin \mathcal{C}_0.$
\end{proof}

\section{Zero divisors and Topological divisors of zero in the disk algebra $\mathscr{A}(\mathbb{D})$}
In this section, we show that zero is the only zero divisor in the disk algebra $\mathscr{A}(\mathbb{D}).$ We also give a class of singular elements in $\mathscr{A}(\mathbb{D}),$ which are not topological divisors of zero. In the end, we give a class of topological divisors of zero in  $\mathscr{A}(\mathbb{D}),$ which are not zero divisors.
\begin{proposition}
	In the disk algebra $\mathscr{A}(\mathbb{D})$ zero is the only zero divisor. 
\end{proposition}
\begin{proof}
	Suppose $0\not\equiv f\in \mathscr{A}(\mathbb{D})$ is a zero divisor. Then there exists $0\neq g\in \mathscr{A}(\mathbb{D})$ such that $(fg)(z)=0$ $\forall\ z\in \mathbb{D}.$ Since $f$ is continuous and $f\not\equiv 0,$ there exists a $z_0 \in \mathbb{D}$ such that $f(z)\neq 0$ in an open disk centered at $z_0,$ say $\mathbb D_1\subseteq \mathbb D.$ Since $(fg)(z)=0$ $\forall\ z\in \mathbb{D}_1.$ It follows that $ g(z)=0$ $\forall\ z\in \mathbb{D}_1$. By Identity principle, $g(z)=0$ $\forall z\in \bar{\mathbb{D}}.$ Thus a non-zero element in $\mathscr{A}(\mathbb{D})$ can not be a zero divisor.
\end{proof}
\begin{remark}
	Every topological divisor is a singular element but the following lemma shows that the converse is not true.
\end{remark}
\begin{lemma}\emph{(\cite{Hoffman,Garcia})}	For a finite sequence $z_1,z_2,...,z_n$ in $\mathbb{D}$ and $\gamma\in \mathbb{T}$, let $$B(z)=\gamma\prod_{i=1}^{n}\frac{z-z_i}{1-\bar{z_i}z}$$ be a finite Blaschke product. Then $B\in \mathscr{A}(\mathbb{D})$ is a singular element but not a topological divisor of zero.
\end{lemma}
\begin{proof}
	Clearly $B\in \mathscr{A}(\mathbb{D})$ and $|B(z)|=1$ for all $z\in \mathbb{T}.$ By the Maximum Modulus Principle, for every $f\in \mathscr{A}(\mathbb{D})$, we have 
	\begin{equation}\label{modeq1}
		\|B f\| = \sup_{z\in \bar{\mathbb{D}}} |B(z)(f(z))| = \max_{z\in \mathbb{T}}|B(z)||f(z)|= \|f\|.
	\end{equation}
	$B$ is a singular element in $\mathscr{A}(\mathbb{D}),$ since $B(z_k)=0$ for each $k=1,2,...,n.$ We now assert that $B$ is not a topological divisor of zero.	
	%Suppose on contrary
	Indeed, if there exists a   sequence   $(g_n)_{n=1}^\infty$ in $\mathscr{A}(\mathbb{D})$ such that $B g_n \to 0$ as $n \to \infty$, then from \eqref{modeq1}, we have 
	$$\|B g_n\|= \|g_n\| \quad \forall\ n\in \mathbb{N}.$$
	Hence $(g_n)_{n=1}^\infty$ must converge to $0.$ Therefore $B$ can not be a topological divisor of zero.
\end{proof}	

\begin{Theorem}
	Let $\mathcal{A}=\mathscr{A}(\mathbb{D})$ be the disk algebra. Let $f(z)=\left(\frac{z-z_0}{2}\right)$ for some $z_0 \in \mathbb{C}$. Then $f$ is topological divisor of zero in $\mathcal{A}$ if and only if $|z_0|=1.$
\end{Theorem}
\begin{proof}
	Suppose $z_0\in \mathbb{T}.$ Define $f_n(z)=\left(\frac{z+z_0}{2}\right)^n\ \mbox {for each}\ n\in \mathbb{N}$. Since  $ z_0\in \mathbb{T}$, we have
	\begin{align*}
		\ f_n\in \mathcal{A} \mbox{ and } |f_n(z_0)|=|z_0^n|=|z_0|^n=1 \quad \forall\ n\in \mathbb{N}.
	\end{align*}
	Therefore
	$ \|f_n\|=1 \quad \forall\ n \in \mathbb{N}.$
	Now note that
	\begin{equation*}%\label{Toeq11}
		ff_n(z)=\left(\frac{z-z_0}{2}\right)\left(\frac{z+z_0}{2}\right)^n,
	\end{equation*}
	and each $z\in \mathbb{T} $ is of the form  $z=e^{i\theta}$ for some  $\theta \in [0,2\pi]$.
	So $z_0=e^{i\theta_0}$  for some  $\theta_0 \in[0,2\pi].$
	Thus, for each $z\in \mathbb{T}$, we have, 
	
	\begin{align*}
		\frac{z-z_0}{2}&=\frac{e^{i\theta}-e^{i\theta_0}}{2}=ie^{i\left(\frac{\theta+\theta_0}{2}\right)}\sin\left(\frac{\theta-\theta_0}{2}\right),\\ \frac{z+z_0}{2}&=\frac{e^{i\theta}+e^{i\theta_0}}{2}=e^{i\left(\frac{\theta+\theta_0}{2}\right)}\cos(\frac{\theta-\theta_0}{2}).
	\end{align*}
	Therefore $f(z)=ie^{i\left(\frac{\theta+\theta_0}{2}\right)}\sin\left(\frac{\theta-\theta_0}{2}\right)$
	and  $f_n(z)=\left(e^{i\left(\frac{\theta+\theta_0}{2}\right)}\cos\left(\frac{\theta-\theta_0}{2}\right)\right)^n.$\\
	This implies that $|ff_n(z)|=\left|\sin\left(\frac{\theta-\theta_0}{2}\right)\cos^n\left(\frac{\theta-\theta_0}{2}\right)\right|.$
	A simple computation shows that
	$$\|ff_n\|=\frac{1}{\sqrt{1+n}}\left(\sqrt{\frac{n}{n+1}}\right)^{n}.$$\\
	Hence
	$\|ff_n\|=\frac{1}{\sqrt{1+n}}\left(\sqrt{\frac{n}{n+1}}\right)^{n} \to 0 \mbox{ as } n\to\infty.$ 
	Hence $f$ is a topological divisor of zero in $\mathcal{A}.$
	
	Now suppose $z_0 \notin \mathbb{T}.$ Let $r=|z_0|<1.$ We will show that $f$ is not a topological divisor of zero in $\mathcal{A}.$
	
	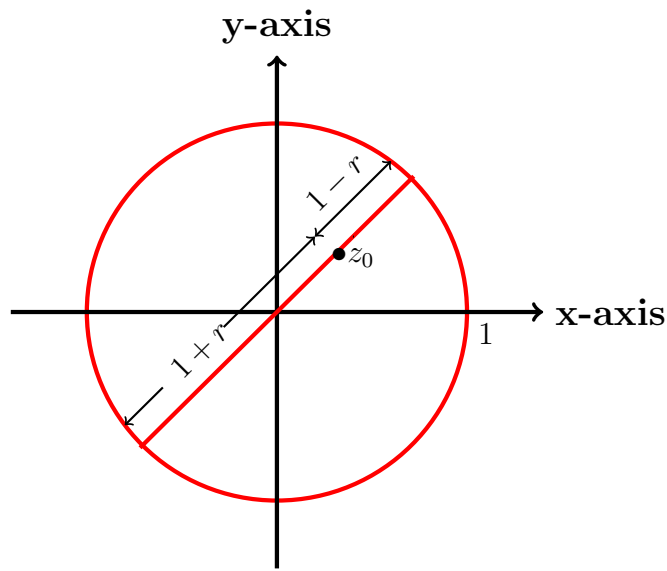
\begin{figure}[H]
		\begin{adjustbox}{center}
			\begin{tikzpicture}
				\node[circle,label=above:{},label=right:{},draw=red!100,minimum width= 5cm,ultra thick](a){};
				\draw[->,ultra thick] (-3.5,0) -- (3.5,0) node[right] {\textbf{\large x-axis}};
				\draw[->,ultra thick] (0,-3.4) -- (0,3.4) node[above] {\textbf{\large y-axis}};
				\draw[color=red,ultra thick][domain=-1.8:1.8]   plot (\x,{\x})   node[right] {};
				\draw[dashed, thick,black] (1, 1) -- node[below=]{$\bullet z_0$} (1,1) ;
				\draw[dashed, thick,black] (2, 0) -- node[below=]{$1$} (3.5,0) ;
				\draw[<-, thick,black] (-2, -1.5) -- node[pos=2,rotate=45]{$1+r$} (-1.5,-1) ;
				\draw[->, thick,black] (-0.7,-0.2) --  (0.5,1) ;
				\draw[<->, thick,black] (0.5,1) -- node[above=, rotate=45]{$1-r$} (1.5,2) ;
			\end{tikzpicture}
		\end{adjustbox}
		\caption{Bounds for $|f(z)|$}
		\label{f2}
	\end{figure}
	From FIGURE \ref{f2}, observe that  $(1-r)<|f(z)|<(1+r) \quad \forall\ z\in \mathbb{T}$.
	
	Suppose there exists a sequence $(f_n)_{n=1}^\infty$ in $\mathcal{A}$ such that $ff_n \to 0$ as $n \to \infty.$
	Since 
	$\|ff_n\|=\sup_{z\in \bar{\mathbb{D}}}{|f(z)f_n(z)|}.$
	Therefore $$(1-r)|f_n(z)|\leq\|ff_n\| \quad \forall\ n\in \mathbb{N} \text{ and } z\in \bar{\mathbb{D}}.$$
	Hence $(1-r)\|f_n\|\leq\|ff_n\|\to 0$ as $n\to \infty$
	implies that $(1-r)\|f_n\|\to 0$ as $n\to \infty.$
	Therefore $f_n \to 0$ as $n\to \infty.$
	Hence $f$ can not be a topological divisor of zero in $\mathcal{A}$.\\
	A similar argument shows that if $r=|z_0|>1,$ then $f(z)=(\frac{z-z_0}{2})$ is not a topological divisor of zero in $\mathcal{A}.$

\end{proof}

\end{document}